\documentclass[12pt,leqno]{amsart}
\usepackage[latin9]{inputenc}
\usepackage{verbatim}
\usepackage{amsthm}
\usepackage{amstext}
\usepackage{amssymb}
\usepackage{color}

\makeatletter
\numberwithin{equation}{section}
\numberwithin{figure}{section}

\usepackage{amsthm}\usepackage{amscd}\usepackage{bm}
\theoremstyle{plain}
\newtheorem{main theorem}{Main Theorem}
\newtheorem{theorem}{Theorem}[section]
\newtheorem{lemma}[theorem]{Lemma}

\newtheorem{claim}[theorem]{Claim}

\theoremstyle{definition}

\newtheorem{remark}[theorem]{Remark}

\newtheorem{problem}[theorem]{Problem}

\newcommand{\diam}{\mathrm{diam}}

\newcommand{\supp}{\mathrm{supp}}
\newcommand{\norm}[1]{\left|\!\left|#1\right|\!\right|}

\makeatother

\begin{document}

\title[A Lipschitz refinement of the Bebutov--Kakutani theorem]{A Lipschitz refinement of the Bebutov--Kakutani dynamical embedding theorem}

\author{Yonatan Gutman, Lei Jin,  Masaki Tsukamoto}



\maketitle

\begin{abstract}
We prove that an $\mathbb{R}$-action on a compact metric space
embeds equivariantly in the space of one-Lipschitz functions $\mathbb{R}\to[0,1]$
if its fixed point set can be topologically embedded in the unit interval.
This is a refinement of the classical Bebutov--Kakutani theorem (1968).
\end{abstract}

\section{Introduction}

The purpose of this short paper is to refine a classical theorem of
Bebutov \cite{Bebutov} and Kakutani \cite{Kakutani} on dynamical systems.
We call $(X,T)$ a \textbf{flow} if $X$ is a compact metric space and
\[ T:\mathbb{R}\times X\to X, \quad (t, x) \mapsto T_t x \]
is a continuous action of $\mathbb{R}$.
We define $\mathrm{Fix}(X,T)$ (sometimes abbreviated to $\mathrm{Fix}(X)$) as the set of $x\in X$ satisfying $T_t x=x$ for all $t\in \mathbb{R}$.
We define $C(\mathbb{R})$ as the space of continuous maps $\varphi: \mathbb{R}\to [0,1]$.
It is endowed with the topology of uniform convergence over compact subsets of $\mathbb{R}$, namely the topology given by the
distance
\begin{equation} \label{eq: distance on C(R)}
      \sum_{n=1}^\infty 2^{-n} \max_{|t|\leq n} |\varphi(t)-\psi(t)|, \quad (\varphi, \psi\in C(\mathbb{R})).
\end{equation}
The group $\mathbb{R}$ continuously acts on it by the translation:
\begin{equation} \label{eq: translation action}
 \mathbb{R}\times C(\mathbb{R}) \to C(\mathbb{R}), \quad (s, \varphi(t))\mapsto \varphi(t+s).
\end{equation}
A continuous map $f:X\to C(\mathbb{R})$ is called an \textbf{embedding of a flow} $(X,T)$ if $f$ is an $\mathbb{R}$-equivariant topological embedding.
Bebutov \cite{Bebutov} and Kakutani \cite{Kakutani} found that the $\mathbb{R}$-action on $C(\mathbb{R})$ has the
following remarkable ``universality'':

\begin{theorem}[Bebutov--Kakutani]
A flow $(X,T)$ can be equivariantly embedded in $C(\mathbb{R})$ if and only if $\mathrm{Fix}(X,T)$ can be topologically embedded in
the unit interval $[0,1]$.
\end{theorem}

The ``only if'' part is trivial because the set of fixed points of $C(\mathbb{R})$ is homeomorphic to $[0,1]$.
So the main statement is the ``if'' part.

Although the Bebutov--Kakutani theorem is clearly a nice theorem, it has one drawback:
The space $C(\mathbb{R})$ is not compact (nor locally compact).
So it is not a ``flow'' in the above definition.
This poses the following problem:

\begin{problem}  \label{problem: main problem}
   Is there a \textit{compact} invariant subset of $C(\mathbb{R})$ satisfying the same universality?
\end{problem}

The purpose of this paper is to solve this problem affirmatively.
Let $L(\mathbb{R})$ be the set of maps $\varphi:\mathbb{R}\to [0,1]$ satisfying the one-Lipschitz condition:
\begin{equation}  \label{eq: Lipschitz}
  \forall s, t\in \mathbb{R}: \quad |\varphi(s)-\varphi(t)| \leq |s-t|.
\end{equation}
$L(\mathbb{R})$ is a subset of $C(\mathbb{R})$.
It is compact with respect to the distance (\ref{eq: distance on C(R)}) by Ascoli--Arzela's theorem.
The $\mathbb{R}$-action (\ref{eq: translation action}) preserves $L(\mathbb{R})$.
So it becomes a flow.
Our main result is the following.
This solves \cite[Question 4.1]{GJ}.

\begin{theorem}   \label{theorem: Lipschitz Bebutov--Kakutani}
A flow $(X,T)$ can be equivariantly embedded in $L(\mathbb{R})$ if and only if $\mathrm{Fix}(X,T)$ can be topologically embedded
in the unit interval $[0,1]$.
\end{theorem}

As in the case of the Bebutov--Kakutani theorem, the ``only if'' part is trivial because
the fixed point set
$\mathrm{Fix}(L(\mathbb{R}))$ is homeomorphic to $[0,1]$.
Since $L(\mathbb{R})$ is compact,
it is a more reasonable choice of such a ``universal flow''.

The proof of Theorem \ref{theorem: Lipschitz Bebutov--Kakutani} is based on the techniques
originally used in the proof of the Bebutov--Kakutani theorem (in particular, the idea of \textit{local section}).
A main new ingredient is the topological argument given in Section \ref{section: topological preparations}, which
has some combinatorial flavor.

\begin{remark}
Problem \ref{problem: main problem} asks us to find a universal flow \textit{smaller than} $C(\mathbb{R})$.
If we look for a universal flow \textit{larger than} $C(\mathbb{R})$, then it is much easier to find an example.
Let $L^\infty(\mathbb{R})$ be the set of $L^\infty$-functions $\varphi: \mathbb{R}\to [0,1]$. (We identify two functions
which are equal to each other almost everywhere.)
We consider the weak$^*$ topology on it. Namely a sequence $\{\varphi_n\}$ in $L^\infty(\mathbb{R})$ converges to
$\varphi\in L^\infty(\mathbb{R})$ if for every $L^1$-function $\psi:\mathbb{R}\to \mathbb{R}$
\[  \lim_{n\to \infty} \int_{\mathbb{R}} \varphi_n(t) \psi(t) \, dt  = \int_{\mathbb{R}} \varphi(t) \psi(t) \, dt. \]
Then $L^\infty(\mathbb{R})$ is compact and metrizable by Banach--Alaoglu's theorem and the separability of the space of $L^1$-functions,
respectively.
The group $\mathbb{R}$ acts continuously on it by translation.
So it becomes a flow.
Note that $\mathrm{Fix}\left(L^\infty(\mathbb{R})\right)$ is homeomorphic to $[0,1]$ and that
the natural inclusion map $C(\mathbb{R}) \subset L^\infty(\mathbb{R})$ is an equivariant continuous injection.
Then the Bebutov--Kakutani theorem implies the universality of $L^\infty(\mathbb{R})$:
A flow $(X,T)$ can be equivariantly embedded in $L^\infty(\mathbb{R})$ if and only if $\mathrm{Fix}(X,T)$ can be
topologically embedded in $[0,1]$.
\end{remark}

\vspace{0.3cm}

\textbf{Acknowledgement.}
This paper was written when the third named author stayed in the Einstein Institute of Mathematics in the Hebrew University of
Jerusalem.
He would like to thank the institute for its hospitality.
Y.G. was partially supported by the Marie Curie grant PCIG12-GA-2012-334564.
Y.G. and L.J. were partially supported by the National Science Center (Poland) grant 2013/08/A/ST1/00275.
M.T. was supported by John Mung Program of Kyoto University.

\medskip

\section{Topological preparations}  \label{section: topological preparations}

Let $a$ be a positive number.
We define $L[0,a]$ as the space of maps $\varphi:[0,a]\to [0,1]$ satisfying
\[  \forall s, t\in [0,a]: \quad |\varphi(s)-\varphi(t)| \leq |s-t|.  \]
$L[0,a]$ is endowed with the distance
$\norm{\varphi-\psi}_\infty = \max_{0 \leq t\leq a} |\varphi(t)-\psi(t)|$.
We define $F_L[0, a] \subset L[0,a]$ as the space of constant functions $\varphi:[0,a]\to [0,1]$, which is homeomorphic
to $[0,1]$.

Let $(X,d)$ be a compact metric space.
We define $C\left(X,L[0,a]\right)$ as the space of continuous maps $f:X\to L[0,a]$, which is
endowed with the distance
\[ \max_{x\in X} \norm{f(x)-g(x)}_\infty. \]

\begin{lemma} \label{lemma: avoiding constant functions}
Let $f\in C\left(X, L[0,a]\right)$ and suppose there exists $0<\tau<1$ satisfying
\begin{equation} \label{eq: tau-Lipschitz}
    \forall x\in X, \forall s, t\in [0,a]: \quad |f(x)(s)-f(x)(t)| \leq \tau |s-t|.
\end{equation}
Then for any $\delta>0$ there exists $g\in C\left(X, L[0,a]\right)$ satisfying
   \begin{enumerate}
      \item $\max_{x\in X} \norm{f(x)-g(x)}_\infty <\delta$.
      \item $g(x)(0)=f(x)(0)$ and $g(x)(a)=f(x)(a)$ for all $x\in X$.
      \item $g(X)\cap F_L[0,a] = \emptyset$.
   \end{enumerate}
\end{lemma}

\begin{proof}
We take $0<b<c<a$ satisfying $b=a-c < \delta/4$.
We take an open covering $\{U_1,\dots,U_M\}$ of $X$ satisfying
\begin{equation} \label{eq: open cover}
   \forall 1\leq m\leq M: \quad  \diam f(U_m) < \min \left(\frac{\delta}{4}, \frac{(1-\tau)b}{2}\right).
\end{equation}
We take a point $p_m\in U_m$ for each $m$.
We choose a natural number $N$ satisfying
\[  N > M, \quad \Delta  \overset{\mathrm{def}}{=} \frac{c-b}{N-1} < \frac{\delta}{4}. \]
We divide the interval $[b,c]$ into $(N-1)$ intervals of length $\Delta$:
\[  b= a_1<a_2<\dots<a_{N}=c, \quad a_{n+1}-a_n = \Delta\quad (\forall 1\leq n\leq N-1). \]
Set $A= \{a_1,\dots, a_N\}$ and define a vector $e\in \mathbb{R}^A$ by
$e=(1,1,\dots, 1)$.
Notice that $f(p_m)|_{A}$ is an element of $[0,1]^{A}$.
Since $N>M$ we can choose $u_1,\dots, u_M\in [0,1]^A$ satisfying
\begin{enumerate}
   \item $|f(p_m)(a_n) - u_m(a_n)| < \min\left(\delta/4, (1-\tau)b/2\right)$ for all $1\leq m\leq M$ and $1\leq n\leq N$.
   \item $|u_m(a_{n+1})-u_m(a_n)| < \Delta$ for all $1\leq m\leq M$ and $1\leq n\leq N-1$.
   \item The $(M+1)$ vectors $e, u_1,\dots, u_M$ are linearly independent.
\end{enumerate}

Let $\{h_m\}_{m=1}^M$ be a partition of unity on $X$ satisfying $\supp \, h_m\subset U_m$ for all $m$.
For $x\in X$ we define a piecewise linear function $g(x):[0,a]\to [0,1]$ as follows.
(We set $a_0=0$ and $a_{N+1}=a$.)
\begin{itemize}
   \item $g(x)(0) = f(x)(0)$ and $g(x)(a)= f(x)(a)$.
   \item $g(x)(a_n) = \sum_{m=1}^M h_m(x) u_m(a_n)$ for $1\leq n\leq N$.
   \item We extend $g(x)$ linearly. Namely, for $t= (1-\lambda) a_n + \lambda a_{n+1}$ with $0\leq \lambda\leq 1$ and
           $0\leq n\leq N$ we set $g(x)(t) = (1-\lambda) g(a_n) + \lambda g(a_{n+1})$.
\end{itemize}

\begin{claim}
 $g(x)\in L[0,a]$ and $\norm{g(x)-f(x)}_\infty < \delta$.
\end{claim}

\begin{proof}

For proving $g(x)\in L[0,a]$ it is enough to show $|g(x)(a_{n+1})-g(x)(a_n)| \leq |a_{n+1}-a_n|$ for all $0\leq n\leq N$.
For $1\leq  n\leq N-1$, this is a direct consequence of the property (2) of $u_m$.
So we consider the case of $n=0$. (The case of $n=N$ is the same).
\begin{equation*}
  \begin{split}
   |g(x)(b) - f(x)(0)| \leq & \sum_{m=1}^M h_m(x) |u_m(b)-f(p_m)(b)| + \sum_{m=1}^M h_m(x) |f(p_m)(b) - f(x)(b)| \\
                                  & + |f(x)(b)-f(x)(0)|.
  \end{split}
\end{equation*}
We apply to each term of the right-hand side
the property (1) of $u_m$, $\diam f(U_m) < (1-\tau)b/2$ in (\ref{eq: open cover}) and
$|f(x)(b)-f(x)(0)| \leq \tau b$ in (\ref{eq: tau-Lipschitz})
respectively. Then this is bounded by
\[  \frac{(1-\tau)b}{2} + \frac{(1-\tau)b}{2} + \tau b = b. \]
This proves $g(x)\in L[0,a]$.

Next we show $|g(x)(a_n) - f(x)(a_n)| < \delta/2$ for all $0\leq n\leq N+1$.
For $n=0, N+1$, this is trivial. For $1\leq n\leq N$, we can bound $|g(x)(a_n)-f(x)(a_n)|$ from above by
\begin{equation*}
   \begin{split}
    & \sum_{m=1}^M h_m(x) |u_m(a_n)-f(p_m)(a_n)| + \sum_{m=1}^M h_m(x) |f(p_m)(a_n)-f(x)(a_n)|  \\
    &<  \frac{\delta}{4}+ \frac{\delta}{4} = \frac{\delta}{2} \quad
      \left(\text{by the property (1) of $u_m$ and $\diam f(U_m) < \frac{\delta}{4}$ in (\ref{eq: open cover})}\right).
   \end{split}
\end{equation*}
Finally, let $a_n<t<a_{n+1}$. We can bound $|g(x)(t)-f(x)(t)|$ by
\begin{equation*}
   \begin{split}
     & |g(x)(t)-g(x)(a_n)| + |g(x)(a_n)-f(x)(a_n)| + |f(x)(a_n)-f(x)(t)|  \\
     & <  2(a_{n+1}-a_n) + \frac{\delta}{2}  \quad (\text{by $f(x), g(x)\in L[0,a]$}) \\
     & < \delta \quad  \left(\text{by $a_{n+1}-a_n \leq \max(b, \Delta) < \frac{\delta}{4}$}\right).
   \end{split}
\end{equation*}
\end{proof}

For every $x\in X$, the function $g(x):[0,a]\to [0,1]$ is a non-constant function because
\[ g(x)|_{\Lambda} = \sum_{m=1}^M h_m(x) u_m \not\in \mathbb{R} e \quad (\text{by the property (3) of $u_m$}). \]
This proves the statement.
\end{proof}

We need two lemmas on linear algebra.
For $u=(x_1,\dots, x_{n+1})\in \mathbb{R}^{n+1}$ we set
\[ Du = (x_2-x_1,x_3-x_2,\dots, x_{n+1}-x_{n}) \in \mathbb{R}^{n}. \]

\begin{lemma}  \label{lemma: linear algebra 1}
Let $l\geq m+1$ and set $e=(\underbrace{1,1,\dots, 1}_l) \in \mathbb{R}^{l}$.
The set of $(u_1,\dots, u_m) \in \mathbb{R}^{l+1}\times \dots \times \mathbb{R}^{l+1} = \left(\mathbb{R}^{l+1}\right)^m$ such that
\begin{equation} \label{eq: e and Du}
    \text{the vectors $e, Du_1, Du_2, \dots, Du_{m}$ are linearly independent}
\end{equation}
is open and dense in $\left(\mathbb{R}^{l+1}\right)^{m}$.
\end{lemma}

\begin{proof}
The condition (\ref{eq: e and Du}) defines a Zariski open set in $\left(\mathbb{R}^{l+1}\right)^{m}$.
So it is enough to show that the set is non-empty because a non-empty Zariski open set is always dense in the Euclidean topology.
We set
\[  u_i = (\underbrace{-1, \dots, -1}_{i}, \, \underbrace{0, \dots, 0}_{l+1-i}) , \quad (1\leq i\leq m). \]
Then
\[ Du_i = (\underbrace{0, \dots, 0}_{i-1},\, 1, \, \underbrace{0, \dots, 0}_{l-i}). \]
The vectors $e, Du_1,\dots Du_{m}$ are linearly independent.
\end{proof}

\begin{lemma}  \label{lemma: linear algebra 2}
Let $n > l \geq 2m$.
The set of $(u_1,\dots, u_m) \in \mathbb{R}^n\times \dots \times \mathbb{R}^n = \left(\mathbb{R}^n\right)^m$ such that,
for any integer $\alpha$ with $2\leq \alpha\leq n-l+1$,
 \begin{equation}  \label{eq: pick up linear independence}
    \text{$u_1|_1^l ,\> u_1|_{\alpha}^{\alpha+l-1} ,\> u_2|_1^l, \> u_2|_{\alpha}^{\alpha+l-1}, \dots, u_m|_{1}^l, \> u_m|_{\alpha}^{\alpha+l-1}$
           are linearly independent in $\mathbb{R}^l$}
 \end{equation}
is open and dense in $\left(\mathbb{R}^n\right)^m$.
Here for $u_i = (x_{i1}, \dots, x_{i n})$
\[ u_i|_{1}^l = (x_{i 1},\dots, x_{i l}), \quad
   u_i|_\alpha^{\alpha+l-1} =  (x_{i, \alpha},\dots, x_{i, \alpha+ l-1}). \]
\end{lemma}

\begin{proof}
The condition (\ref{eq: pick up linear independence}) defines a Zariski open set in $\left(\mathbb{R}^n\right)^m$.
Hence it is enough to show that for
each $2 \leq \alpha \leq n-l+1$ we can choose $(u_1,\dots, u_m) \in  \left(\mathbb{R}^n\right)^m$
satisfying (\ref{eq: pick up linear independence}).

We define $u_i = (x_{i1}, \dots, x_{i n})$ $(1 \leq i\leq m)$ by
\[ x_{ij} = 1 \quad \left(j= i, \alpha + l-i \right), \quad x_{ij} = 0 \quad (\text{otherwise}). \]
Then it is direct to check that these $u_i$ satisfy (\ref{eq: pick up linear independence}).
One can also use a proof from \cite[Lemma 5.5]{Lindenstrauss}.
\end{proof}

\begin{lemma} \label{key lemma}
Let $f\in C\left(X, L[0,a]\right)$ and suppose there exists $0<\tau<1$ satisfying (\ref{eq: tau-Lipschitz}).
Then for any $\delta>0$ there exists $g\in C\left(X, L[0,a]\right)$ satisfying
    \begin{enumerate}
       \item $\max_{x\in X} \norm{f(x)-g(x)}_\infty < \delta$.
       \item $g(x)(0)=f(x)(0)$ and $g(x)(a)=f(x)(a)$ for all $x\in X$.
       \item If $x,y\in X$ and $0\leq \varepsilon  \leq a/2$ satisfy
               \[  \forall t\in [0,a-\varepsilon]:  g(x)(t+\varepsilon) = g(y)(t) \]
               then $\varepsilon=0$ and $d(x,y) < \delta$.
    \end{enumerate}
\end{lemma}

\begin{proof}
Except for the use of the above two lemmas on linear algebra,
the proof is close to Lemma \ref{lemma: avoiding constant functions}.
We take $0<b<c<a$ with $b=a-c < \min\left(\delta/4, a/4\right)$.
We take an open covering $\{U_1,\dots,U_M\}$ satisfying $\diam\, U_m < \delta$ and
$\diam f(U_m) < \min\left(\delta/4, (1-\tau)b/2\right)$ for all $1\leq m\leq M$.
Take $p_m\in U_m$ for each $m$.
Let $N\geq 2$ be a natural number and set $\Delta = (c-b)/(N-1)$.
We introduce a partition $b=a_1<a_2<\dots<a_N=c$ by $a_n = b + (n-1)\Delta$.
We set $A= \{a_1,\dots, a_N\}$ and $\Lambda = A\cap [b,a/4] = \{a_1,\dots, a_L\}$.
We also set $e=(\underbrace{1,1,\dots, 1}_L)\in \mathbb{R}^{L}$.
We choose $N$ sufficiently large so that
\[  \Delta < \frac{\delta}{4}, \quad N> L \geq 2M. \]

Since $L\geq 2M\geq M+1$, by using Lemmas \ref{lemma: linear algebra 1} and \ref{lemma: linear algebra 2},
we can choose $u_1,\dots, u_M\in [0,1]^A$ satisfying
\begin{enumerate}
   \item $|f(p_m)(a_n) - u_m(a_n)| < \min\left(\delta/4, (1-\tau)b/2\right)$ for all $1\leq m\leq M$ and $1\leq n\leq N$.
   \item $|u_m(a_{n+1})-u_m(a_n)| < \Delta$ for all $1\leq m\leq M$ and $1\leq n\leq N-1$.
   \item Define $D_L u_m = (u_m(a_2)-u_m(a_1),\dots, u_m(a_{L+1})-u_m(a_L))\in \mathbb{R}^L$. Then the $(M+1)$ vectors
           $e, D_L u_1, \dots, D_L u_M$ in $\mathbb{R}^L$ are linearly independent.
   \item For any $\varepsilon > 0$ with $\varepsilon + \Lambda \subset A$,
    \begin{equation*}
    \text{$u_1|_{\Lambda},\> u_1|_{\varepsilon+\Lambda},\>
     u_2|_{\Lambda},\> u_2|_{\varepsilon+\Lambda}, \dots, u_m|_{\Lambda}, \> u_m|_{\varepsilon+\Lambda}$
           are linearly independent in $\mathbb{R}^\Lambda$}
    \end{equation*}
\end{enumerate}
For $x\in X$ we define $g(x):[0,a]\to [0,1]$ in the same way as in the proof of Lemma \ref{lemma: avoiding constant functions}.
Namely, we set $g(x)(0)= f(x)(0)$, $g(x)(a)= f(x)(a)$ and $g(x)(a_n) = \sum_{m=1}^M h_m(x) u_m(a_n)$ for $1\leq n\leq N$, where
$\{h_m\}$ is a partition of unity satisfying $\supp \, h_m\subset U_m$.
We extend $g(x)$ to $[0,a]$ by linearity.
It follows that $g(x)\in L[0,a]$ and $\norm{g(x)-f(x)}_\infty < \delta$ as before.
We need to check the property (3) of the statement.
Suppose there exist $x,y\in X$ and $0\leq \varepsilon \leq a/2$ satisfying $g(x)(t+\varepsilon) = g(y)(t)$ for all $0\leq t\leq a-\varepsilon$.

First we show $\varepsilon + \Lambda\subset A$.
Otherwise, $(\varepsilon +\Lambda) \cap A = \emptyset$.
Then it follows from the piecewise linearity that the function $g(y)(t)$ becomes differentiable at every $t\in \Lambda$, which implies
\[ g(y)(a_{n+1}) - g(y)(a_n) = g(y)(a_{n+2})-g(y)(a_{n+1})  \quad (1\leq n\leq L-1), \]
and hence
\[   \sum_{m=1}^M h_m(y) \left(u_m(a_{n+1}) - u_m(a_n)\right)  = \sum_{m=1}^M h_m(y)\left(u_m(a_{n+2})-u_m(a_{n+1})\right)
      \quad (1\leq n\leq L-1). \]
This means that $\sum_{m=1}^M h_m(y) D_L u_m \in \mathbb{R} e$, which contradicts the property (3) of $u_m$.
So we must have $\varepsilon +\Lambda \subset A$.

The equation $g(x)(t+\varepsilon) = g(y)(t)$ $(0\leq t\leq a-\varepsilon)$ implies
\[  \sum_{m=1}^M h_m(x) u_m|_{\varepsilon +\Lambda} = \sum_{m=1}^M h_m(y) u_m|_{\Lambda}.  \]
It follows from the property (4) of $u_m$ that $\varepsilon =0$ and $h_m(x)=h_m(y)$ for all $1\leq m\leq M$.
Then $x,y\in U_m$ for some $m$ and hence
$d(x,y) \leq \diam\, U_m < \delta$.
\end{proof}

\medskip

\section{Proof of Theorem \ref{theorem: Lipschitz Bebutov--Kakutani}}

Let $(X,T)$ be a flow. Set $F=\mathrm{Fix}(X,T)$.
We define $F_L = \mathrm{Fix}\left(L(\mathbb{R})\right)$.
Namely $F_L$ is the space of constant maps $\varphi:\mathbb{R}\to [0,1]$, which is
homeomorphic to $[0,1]$.
Suppose there exists a topological embedding $h:F\to F_L$.
We would like to show that there exists an equivariant embedding $f:X\to L(\mathbb{R})$ with $f|_F=h$.
We define $C_{T,h}\left(X,L(\mathbb{R})\right)$ as the space of equivariant continuous maps $f:X\to L(\mathbb{R})$
satisfying $f|_{F} = h$, which is endowed with the compact-open topology.
For $f\in C_{T,h}\left(X, L(\mathbb{R})\right)$ we define $\mathrm{Lip}(f)$ as the supremum of
\[  \frac{|f(x)(t)-f(x)(s)|}{|s-t|} \]
over all $x\in X$ and $s,t\in \mathbb{R}$ with $s\neq t$.

\begin{lemma} \label{lemma: nonempty}
The space $C_{T,h}\left(X, L(\mathbb{R})\right)$ is not empty.
Moreover for any $\delta>0$ there exists $f\in C_{T,h}\left(X, L(\mathbb{R})\right)$ satisfying
$\mathrm{Lip}(f)\leq \delta$.
\end{lemma}

\begin{proof}
Consider the map
\[  F\ni x\to h(x)(0) \in [0,1]. \]
By the Tietze extension theorem, we can extend this function to
a continuous map $h_0:X\to [0,1]$.
Let $\varphi: \mathbb{R}\to [0,1]$ be a smooth function satisfying
\[  \int_{-\infty}^\infty \varphi(t)\, dt = 1, \quad \int_{-\infty}^\infty \varphi'(t) \, dt \leq \min\left(1, \delta\right). \]
For $x\in X$ we define $f(x): \mathbb{R} \to [0,1]$ by
\[  f(x)(t) = \int_{-\infty}^\infty \varphi(t-s) h_0(T_s x) \, ds. \]
Then $|f(x)'(t)|\leq \min(1,\delta)$ and $f=h$ on $F$. Hence $f\in C_{T,h}\left(X, L(\mathbb{R})\right)$ and $\mathrm{Lip}(f)\leq \delta$.
\end{proof}

We borrow the next lemma from Auslander \cite[p. 186, Corollary 6]{Auslander}.

\begin{lemma} \label{lemma: local section}
Let $p\in X\setminus F$.
There exist $a >0$ and a closed set $S\subset X$ containing $p$ such that the map
\begin{equation} \label{eq: local section}
  [-a, a] \times S\to X, \quad (t, x)\mapsto T_t x
\end{equation}
is a continuous injection whose image contains an open neighborhood of $p$ in $X$.
We call $(a, S)$ a \textbf{local section} around $p$ and denote
the image of (\ref{eq: local section}) by $[-a,a] \cdot S$.
\end{lemma}

\begin{proof}
We explain the proof for the convenience of readers.
We can find $c<0$ and a continuous function $h:X\to [0,1]$
satisfying $T_c p \not\in \supp\, h$ and $h=1$ on a neighborhood of $p$.
We define $f:X\to \mathbb{R}$ by
\[  f(x) = \int_c^0 h(T_t x) dt. \]
We choose $0<a<|c|$ and a closed neighborhood $A$ of $p$ satisfying
\[  \bigcup_{|t|\leq a} T_t(A) \subset \{h=1\}, \quad
      \bigcup_{|t|\leq a} T_{t+c}(A) \cap \supp\, h = \emptyset.  \]
It follows that $f\left(T_t x\right) = f(x) + t$ for any $x\in A$ and $|t|\leq a$.
Set $S= \{x\in A|\, f(x)=f(p)\}$.
Then $(a,S)$ becomes a local section.
Indeed if $x,y\in S$ and $s,t\in [-a,a]$ satisfy $T_s x = T_t y$, then
$s+f(p) = f(T_s x)= f(T_t y) = t+f(p)$ and hence $s=t$ and $x=y$.
Thus the map (\ref{eq: local section}) is injective.
We take $0<b<a$ and an open neighborhood $U$ of $p$ satisfying $\bigcup_{|t|<b} T_t(U) \subset A$.
Then the set $[-a,a]\cdot S$ contains
\begin{equation} \label{eq: open neighborhood of p in the local section lemma}
    \{x\in U|\, -b<f(x)-f(p)<b\}
\end{equation}
because if $x\in U$ satisfies $t \overset{\mathrm{def}}{=} f(x)-f(p) \in (-b,b)$ then
$f(T_{-t} x) = f(x)-t = f(p)$ (i.e. $T_{-t}x\in S$) and $x = T_t(T_{-t} x) \in [-a,a]\cdot S$.
The set (\ref{eq: open neighborhood of p in the local section lemma}) is an open neighborhood of $p$.
\end{proof}

\begin{lemma}  \label{lemma: separating F}
For any point $p\in X\setminus F$ there exists a closed neighborhood $A$ of $p$ in $X$ such that the set
\begin{equation} \label{eq: G(A)}
     G(A) = \left\{f\in C_{T,h}\left(X, L(\mathbb{R})\right) |\, f(A) \cap F_L = \emptyset \right\}
\end{equation}
is open and dense in the space $C_{T,h}\left(X, L(\mathbb{R})\right)$.
\end{lemma}

\begin{proof}
Take a local section $(a, S)$ around $p$.
For $x\in X$ we define $H(x)\subset \mathbb{R}$ (the set of \textbf{hitting times}) as the set of $t\in \mathbb{R}$
satisfying $T_t x\in S$. Any two distinct $s,t\in H(x)$ satisfy $|s-t|> a$.
Notice that if $x\in F$ then $H(x)=\emptyset$.
We denote by $\mathrm{Int}\left([-a,a]\cdot S\right)$ the interior of $[-a,a]\cdot S$.
We choose a closed neighborhood $A_0$ of $p$ in $S$ satisfying $A_0\subset \mathrm{Int}\left([-a,a]\cdot S\right)$.
We define a closed neighborhood $A$ of $p$ in $X$ by
\[    A = \bigcup_{|t|\leq a} T_t(A_0). \]
We choose a continuous function $q: S\to [0,1]$ satisfying $q=1$ on $A_0$ and $\supp \, q \subset \mathrm{Int}\left([-a,a]\cdot S\right)$.

The set $G(A)$ defined in (\ref{eq: G(A)}) is obviously open. So it is enough to prove that it is dense.
Take $f\in C_{T,h}\left(X, L(\mathbb{R})\right)$ and $0<\delta<1$.
By Lemma \ref{lemma: nonempty} we can find $f_0\in C_{T,h}\left(X, L(\mathbb{R})\right)$ satisfying $\mathrm{Lip}(f_0)\leq 1/2$.
We define $f_1\in C_{T,h}\left(X, L(\mathbb{R})\right)$ by
\[ f_1(x)(t) = (1-\delta) f(x)(t) + \delta f_0(x)(t). \]
It follows $\mathrm{Lip}(f_1) \leq 1-\delta/2 < 1$.
We apply Lemma \ref{lemma: avoiding constant functions} to the map
\[ X\ni x\mapsto f_1(x)|_{[0,a]} \in L[0,a]. \]
Then we find $g\in C\left(X, L[0,a]\right)$ satisfying
\begin{enumerate}
   \item $|g(x)(t)-f_1(x)(t)| < \delta$ for all $x\in X$ and $0\leq t\leq a$.
   \item $g(x)(0)= f_1(x)(0)$ and $g(x)(a)= f_1(x)(a)$ for all $x\in X$.
   \item $g(X) \cap F_L[0,a] = \emptyset$.
\end{enumerate}
We set $u(x)(t) = g(x)(t)-f_1(x)(t)$ for $x\in X$ and $0\leq t\leq a$.
We define $g_1\in C_{T,h}\left(X, L(\mathbb{R})\right)$ as follows: Let $x\in X$.
\begin{itemize}
   \item For each $s\in H(x)$, we set
   \[  g_1(x)(t) = f_1(x)(t) + q(T_s x)\cdot  u(T_s x)(t-s) \quad \text{for $t\in [s,s+a]$}.  \]
   \item For $t\in \mathbb{R}\setminus \bigcup_{s\in H(x)} [s,s+a]$, we set $g_1(x)(t) = f_1(x)(t)$.
\end{itemize}
This satisfies
\[  |g_1(x)(t)-f(x)(t)| \leq |g_1(x)(t)-f_1(x)(t)| + |f_1(x)(t)-f(x)(t)| \leq 3\delta \]
for all $x\in X$ and $t\in \mathbb{R}$.
If $x\in A$ then there exists $s\in [-a,a]$ with $T_s x\in A_0$ and hence
\[  g_1(x)(s+t) = g(T_s x)(t) \quad \text{for $t\in [0,a]$}. \]
It follows from the property (3) of $g$ that the function $g_1(x)$ is not constant.
Thus $g_1\in G(A)$.
Since $f$ and $\delta$ are arbitrary, this proves that $G(A)$ is dense in $C_{T,h}\left(X, L(\mathbb{R})\right)$.
\end{proof}

\begin{lemma} \label{lemma: separating two points}
For any two distinct points $p$ and $q$ in $X\setminus F$ there exist closed neighborhoods $B$ and $C$ of $p$ and $q$ in $X$
respectively such that
the set
\begin{equation} \label{eq: G(B,C)}
    G(B,C) = \left\{f\in C_{T,h}\left(X, L(\mathbb{R})\right) |\, f(B) \cap f(C) = \emptyset \right\}
\end{equation}
is open and dense in $C_{T,h}\left(X, L(\mathbb{R})\right)$.
\end{lemma}

\begin{proof}
Take local sections $(a, S_1)$ and $(a, S_2)$ around $p$ and $q$ respectively.
We can assume that $[-a,a]\cdot S_1$ and $[-a,a]\cdot S_2$ are disjoint with each other.
For $x\in X$ we define $H(x)$ as the set of $t\in \mathbb{R}$ satisfying $T_t x\in S_1\cup S_2$.
We choose closed neighborhoods $B_0$ of $p$ in $S_1$ and $C_0$ of $q$ in $S_2$ respectively
satisfying $B_0\subset \mathrm{Int}\left([-a,a]\cdot S_1\right)$ and $C_0\subset \mathrm{Int}\left([-a,a]\cdot S_2\right)$.
We take a continuous function $\tilde{q}:X\to [0,1]$ satisfying $\tilde{q}=1$ on $B_0\cup C_0$ and
$\supp\, \tilde{q}\subset \mathrm{Int}\left([-a,a]\cdot S_1\right)\cup \mathrm{Int}\left([-a,a]\cdot S_2\right)$.
We define closed neighborhoods $B$ and $C$ of $p$ and $q$ respectively by
\[ B = \bigcup_{|t|\leq a/4} T_t(B_0), \quad C = \bigcup_{|t|\leq a/4} T_t(C_0). \]

The set $G(B,C)$ defined in (\ref{eq: G(B,C)}) is open. We show that it is dense.
Take $f\in C_{T,h}\left(X, L(\mathbb{R})\right)$ and $0<\delta<1$.
We can assume that
\begin{equation} \label{eq: delta}
    \delta < d(B_0,C_0) \overset{\mathrm{def}}{=} \min_{x\in B_0, y\in C_0} d(x,y).
\end{equation}
We define $f_1\in C_{T,h}\left(X, L(\mathbb{R})\right)$ exactly in the same way as in the proof of Lemma \ref{lemma: separating F}.
It satisfies $\mathrm{Lip}(f_1) \leq 1-\delta/2$ and $|f(x)(t)-f_1(x)(t)| \leq 2\delta$ for all $x\in X$ and $t\in \mathbb{R}$.

We apply Lemma \ref{key lemma} to the map
\[  X\ni x\mapsto f_1(x)|_{[0,a]} \in L[0,a].  \]
Then we find $g\in C\left(X, L[0,a]\right)$ satisfying
    \begin{enumerate}
       \item $|g(x)(t)- f_1(x)(t)| < \delta$ for all $x\in X$ and $0\leq t \leq a$.
       \item $g(x)(0)=f_1(x)(0)$ and $g(x)(a)=f_1(x)(a)$ for all $x\in X$.
       \item If $x,y\in X$ and $0\leq \varepsilon \leq a/2$ satisfy
               \[  \forall t\in [0,a-\varepsilon]:  g(x)(t+\varepsilon) = g(y)(t) \]
               then $d(x,y) < \delta$.
    \end{enumerate}
We set $u(x)(t) = g(x)(t)-f_1(x)(t)$ for $x\in X$ and $0\leq t\leq a$.
We define $g_1\in C_{T,h}\left(X, L(\mathbb{R})\right)$ as before.
Namely, for $x\in X$,
\begin{itemize}
   \item For each $s\in H(x)$, we set
   \[  g_1(x)(t) = f_1(x)(t) + \tilde{q}(T_s x)\cdot  u(T_s x)(t-s) \quad \text{for $t\in [s,s+a]$}.  \]
   \item For $t\in \mathbb{R}\setminus \bigcup_{s\in H(x)} [s,s+a]$, we set $g_1(x)(t) = f_1(x)(t)$.
\end{itemize}
This satisfies $|g_1(x)(t)-f(x)(t)| \leq |g_1(x)(t)-f_1(x)(t)| + |f_1(x)(t)-f(x)(t)| \leq 3\delta$.

We would like to show $g_1(B)\cap g_1(C) = \emptyset$.
Suppose $x\in B$ and $y\in C$ satisfy $g_1(x)=g_1(y)$.
There exist $|s_1|\leq a/4$ and $|s_2|\leq a/4$ satisfying $T_{s_1} x\in B_0$ and $T_{s_2} y\in C_0$.
We can assume $s_1\leq s_2$ without loss of generality.
Set $\varepsilon=s_2-s_1 \in [0,a/2]$.
We have
\[  g_1(x)(s_1+t) = g(T_{s_1} x)(t) \text{ and } g_1(y)(s_2+t) = g(T_{s_2} y)(t) \quad
   \text{for $t\in [0,a]$}. \]
$g_1(x)=g_1(y)$ implies that
\[  g(T_{s_1} x) (t+\varepsilon) = g(T_{s_2} y) (t) \quad \text{for $t\in [0, a-\varepsilon]$}. \]
It follows from the property (3) of $g$ that $d(T_{s_1}x, T_{s_2}y) < \delta$.
Since $\delta < d(B_0, C_0) \leq d(T_{s_1}x, T_{s_2}y)$, this is a contradiction.
Therefore $g_1(B)\cap g_1(C) = \emptyset$.
This proves the lemma.
\end{proof}

Now we can prove Theorem \ref{theorem: Lipschitz Bebutov--Kakutani}.
By Lemmas \ref{lemma: separating F} and \ref{lemma: separating two points}, there exist families of closed sets
$\{A_n\}_{n=1}^\infty$, $\{B_n\}_{n=1}^\infty$ and $\{C_n\}_{n=1}^\infty$ of $X\setminus F$ such that
\begin{itemize}
   \item $X\setminus F = \bigcup_{n=1}^\infty A_n$ and
   $(X\setminus F)\times (X\setminus F) \setminus\{(x,x):x\in X\} = \bigcup_{n=1}^\infty B_n\times C_n$.
   \item $G(A_n)$ are open and dense in the space $C_{T,h}\left(X, L(\mathbb{R})\right)$ for all $n\geq 1$.
   \item $G(B_n,C_n)$ are open and dense in the space $C_{T,h}\left(X, L(\mathbb{R})\right)$ for all $n\geq 1$.
\end{itemize}
By the Baire category theorem, the set
\[  \bigcap_{n=1}^\infty G(A_n) \cap \bigcap_{n=1}^\infty G(B_n,C_n) \]
is dense and $G_\delta$ in $C_{T,h}\left(X, L(\mathbb{R})\right)$.
In particular it is not empty.
Any element $f$ in this set gives an embedding of the flow $(X,T)$ in $L(\mathbb{R})$.

\begin{remark}
The proof of the Bebutov--Kakutani theorem in \cite{Kakutani, Auslander} used the idea of ``constructing large derivative''.
It is possible to prove Theorem \ref{theorem: Lipschitz Bebutov--Kakutani} by adapting this idea to
the setting of one-Lipschitz functions.
But this approach seems a bit tricky and less flexible than the proof given above.
The above proof possibly has a wider applicability to different situations
(e.g. other function spaces).
\end{remark}

\vspace{0.5cm}

\address{Yonatan Gutman \endgraf
Institute of Mathematics, Polish Academy of Sciences,
ul. \'{S}niadeckich~8, 00-656 Warszawa, Poland}

\textit{E-mail address}: \texttt{y.gutman@impan.pl}

\vspace{0.5cm}

\address{Lei Jin \endgraf
Institute of Mathematics, Polish Academy of Sciences, ul. \'{S}niadeckich 8, 00-656 Warszawa, Poland}

\textit{E-mail address}: \texttt{jinleim@mail.ustc.edu.cn}

\vspace{0.5cm}

\address{ Masaki Tsukamoto \endgraf
Department of Mathematics, Kyoto University, Kyoto 606-8502, Japan}

\textit{E-mail address}: \texttt{tukamoto@math.kyoto-u.ac.jp}

\textit{Current address}:
Einstein Institute of Mathematics, Hebrew University, Jerusalem 91904, Israel


\medskip


\begin{thebibliography}{999999}





\bibitem[Aus88]{Auslander}
J. Auslander,
Minimal flows and their extensions,
North-Holland, Amsterdam, 1988.


\bibitem[Beb40]{Bebutov}
M.V. Bebutov,
On dynamical systems in the space of continuous functions,
Byull. Moskov. Gos. Univ. Mat. (1940) 2, no.5, 1-52.



\bibitem[GJ16]{GJ}
Y. Gutman and L. Jin,
An explicit compact universal space for real flows,
submitted, 2016.




\bibitem[Kak68]{Kakutani}
S. Kakutani,
A proof of Beboutov's theorem,
J. Differential equations
(1968) 4(2), 194-201.




\bibitem[Lin99]{Lindenstrauss}
E. Lindenstrauss,
Mean dimension, small entropy factors and an embedding theorem,
Publications Math\'ematiques de l'Institut des Hautes \'Etudes Scientifiques
1999, 89(1), 227-262.




\end{thebibliography}
\end{document}